\input ./style/arxiv-general.cfg
\documentclass[aap,MSNbibl,seceqn,dvips]{arximspdf}
\makeatletter
   \@ifpackageloaded{graphicx}{}{\usepackage{graphicx}}
\makeatother
\usepackage{mathbh}

\doi{10.1214/14-AAP1041} 
\volume{25}
\issue{4}
\pubyear{2015}
\firstpage{2013}
\lastpage{2038}
\docsubty{FLA}

\makeatletter
\renewcommand{\mid}{|}
\newcommand{\rrvert}{\vert}
\newcommand{\llvert}{\vert}
\newtheorem{conjecture}{Conjecture}
\newtheorem{open}{Open Problem}
\newproclaim{case}{Case}

\newproclaim{definition}{Definition}
\newtheorem{lemma}{Lemma}
\newtheorem{proposition}{Proposition}
\newproclaim{remark}{Remark}
\newtheorem{theorem}{Theorem}
\def\vk{\mathbf{k}}
\newcommand{\degDist}{\mathrm{p}}
\makeatother

\begin{document}
\begin{frontmatter}

\title{Jigsaw percolation: What social networks can collaboratively
solve a puzzle?}
\runtitle{Jigsaw percolation}

\begin{aug}
\author[A]{\fnms{Charles D.} \snm{Brummitt}\corref{}\thanksref{T1}\ead[label=e1]{brummitt@gmail.com}\ead[label=u1,url]{www.math.ucdavis.edu/\textasciitilde cbrummitt/}},
\author[B]{\fnms{Shirshendu} \snm{Chatterjee}\ead[label=e2]{shirshendu@cims.nyu.edu}\ead[label=u2,url]{www.cims.nyu.edu/\textasciitilde chatterj/}},\\
\author[C]{\fnms{Partha S.} \snm{Dey}\thanksref{T3}\ead[label=e3]{psdey1@gmail.com}\ead[label=u3,url]{www2.warwick.ac.uk/fac/sci/statistics/staff/academic-research/dey/}}
\and
\author[D]{\fnms{David} \snm{Sivakoff}\thanksref{T4}\ead[label=e4]{dsivakoff@stat.osu.edu}\ead[label=u4,url]{www.stat.osu.edu/\textasciitilde dsivakoff/}}
\runauthor{Brummitt, Chatterjee, Dey and Sivakoff}
\affiliation{University of California, Davis,
New York University,\\ 
University of Warwick and
Ohio State University}
\address[A]{C. D. Brummitt\\
Department of Mathematics\\
University of California\\
One Shields Avenue\\
Davis, California 95616\\
USA\\
\printead{e1}\\
\printead{u1}}
\address[B]{S. Chatterjee\\
Courant Institute\\
\quad of Mathematical Sciences\\
New York University\\
251 Mercer Street\\
New York, New York 10012\\
USA\\
\printead{e2}\\
\printead{u2}}
\address[C]{\hspace*{61pt}P. S. Dey\\
\hspace*{61pt}Department of Statistics\\
\hspace*{61pt}University of Warwick\\
\hspace*{61pt}Gibbet Hill Road\\
\hspace*{61pt}Coventry\\
\hspace*{61pt}CV4 7AL\\
\hspace*{61pt}United Kingdom\\
\hspace*{61pt}\printead{e3}\\
\hspace*{61pt}\printead{u3}}
\address{}
\address[D]{D. Sivakoff\\
Department of Statistics\\
\quad and Department of Mathematics\\
Ohio State University\\
1958 Neil Avenue, 404 Cockins Hall\\
Columbus, Ohio 43210\\
USA\\
\printead{e4}\\
\printead{u4}}
\end{aug}
\thankstext{T1}{Supported by the Statistical and Applied Mathematical
Sciences Institute (SAMSI),
the Department of Defense (DoD) through the National Defense Science
and Engineering Graduate Fellowship \mbox{(NDSEG)} Program,
the Defense Threat Reduction Agency Basic Research Award
HDTRA1-10-1-0088 and the Army Research Laboratory Cooperative Agreement
W911NF-09-2-0053.}
\thankstext{T3}{Supported by Simons Postdoctoral Fellowship. Did this
work while at the Courant Institute at New York University.}
\thankstext{T4}{Supported in part by NSF Grant DMS-10-57675 and by
SAMSI. Did this work while at the Duke University Mathematics Department.}

\received{\smonth{9} \syear{2012}}
\revised{\smonth{5} \syear{2014}}

%
\begin{abstract}
We introduce a new kind of percolation on finite graphs called
\emph{jigsaw percolation}. This model attempts to capture networks
of people who innovate by merging ideas and who solve problems by
piecing together solutions. Each person in a social network has a
unique piece of a jigsaw puzzle. Acquainted people with compatible
puzzle pieces merge their puzzle pieces. More generally, groups of
people with merged puzzle pieces merge if the groups know one
another and have a pair of compatible puzzle pieces. The social
network solves the puzzle if it eventually merges all the puzzle
pieces. For an Erd\H{o}s--R\'{e}nyi social network with $n$ vertices
and edge
probability $p_n$, we define the critical value $p_c(n)$ for a
connected puzzle graph to be the $p_n$ for which the chance of
solving the puzzle equals $1/2$. We prove that for the $n$-cycle
(ring) puzzle, $p_c(n) = \Theta(1/\log n)$, and for an arbitrary
connected puzzle graph with bounded maximum degree, $p_c(n) =
O(1/\log n)$ and $\omega(1/n^b)$ for any $b>0$. Surprisingly, with
probability tending to 1 as the network size increases to infinity,
social networks with a power-law degree distribution cannot solve
any bounded-degree puzzle. This model suggests a mechanism for
recent empirical claims that innovation increases with social
density, and it might begin to show what social networks stifle
creativity and what networks collectively innovate.
\end{abstract}

%
\end{frontmatter}

\section{Introduction}
Solving difficult problems and creating new ideas are sometimes
compared to merging the pieces of a puzzle~\cite{Ball2014,JohnsonBook}.
Often these \mbox{breakthroughs} are achieved not by one person working in
isolation but rather by a collection of people who exchange and merge
partial solutions and ideas~\cite{JohnsonBook}.
As a result, the structure of collaboration networks (who collaborates
with whom) can affect the success of the network's creative output,
as found empirically for scientific breakthroughs~\cite
{Chai2011,Gerstein2007,Lambiotte2009} and for hit Broadway
musicals~\cite{Uzzi2005,Uzzi2008}. In business, some companies connect
their employees using internal social networks~\cite{Forbes} and
expertise location systems~\cite{TacitArticle} to match compatible
ideas and expertise. Some companies outsource their most difficult R\&D
problems to leverage knowledge worldwide using services such as \href
{http://www.innocentive.com/}{Innocentive} and \href{http://www.kaggle.com/}{Kaggle}.
Digital tools for massive collaboration are also being used to solve
problems in mathematics~\cite{Gowers2009}, climate change~\cite
{Introne2011} and software design~\cite{Lakhani2010}.

Here we formalize this metaphor of a large group of people
collaboratively solving a puzzle by introducing a new kind of
percolation on finite graphs that aims to model a network of people who
merge compatible ideas into bigger and better ideas.
The model is reminiscent of other models of percolation on graphs, such
as bond percolation~\cite{GrimmettBook} and bootstrap
percolation~\cite{Holroyd2003}, but jigsaw percolation has more
complex dynamics.

Consider a social network of $n$ people with vertex set $V = \{1, 2,\ldots, n\}$, each of whom has a unique ``partial idea'' that could
merge with one or more other partial ideas belonging to other people.
These ``partial ideas'' can be thought of as pieces of a jigsaw puzzle:
an idea is compatible with certain other ideas, just as a piece~of a
jigsaw puzzle can join with certain other puzzle pieces (in the correct
solution of the puzzle). Thus we use ``ideas'' and ``puzzle pieces''
interchangeably. The two networks are:
\begin{itemize}
\item the \emph{people graph} $(V,E_{\mathrm{people}})$, denoting who
knows and communicates with whom;
\item the \emph{puzzle graph} $(V,E_{\mathrm{puzzle}})$, denoting
which ideas are compatible and thus can merge to form a bigger, better idea.
\end{itemize}
In this paper, we assume each person has a unique idea, so there are
$n$ ideas (puzzle pieces), and the system of people and their
compatible ideas is a graph with two sets of edges, $E_{\mathrm
{people}}$ and $E_{\mathrm{puzzle}}$. Allowing a person to have
multiple ideas or multiple people to have the same idea requires two
vertex sets, which we leave for future work; see Section~\ref{discussion}.

Next we propose a natural dynamic for people to merge their compatible
ideas (puzzle pieces). If two people $u,w$ know each other and have
compatible puzzle pieces (i.e., $uw \in E_{\mathrm{people}}\cap
E_{\mathrm{puzzle}}$), then they merge their puzzle pieces. After
$u,w$ merge their puzzle pieces, we say that $u,w$ belong to the same
\emph{jigsaw cluster} \mbox{$U \subseteq V$}. The general rule is that two
jigsaw clusters $U,W$ merge if at least two people (one from each
cluster) know each other, and at least two people (one from each
cluster) have compatible puzzle pieces. More precisely, we say that
jigsaw clusters $U, W$ are \emph{people-adjacent} if $uw \in
E_{\mathrm{people}}$ for some $u \in U, w \in W$. Similarly, $U,W$ are
\emph{puzzle-adjacent} if $u'w' \in E_{\mathrm{puzzle}}$ for some $u'
\in U, w' \in W$. Jigsaw clusters $U,W$ merge if they are both
people-adjacent and puzzle-adjacent.

The motivation for this dynamic is the notion that after merging their
ideas, a~group of people can use any of those ideas to merge with the
ideas of other people whom they know. We illustrate this in Figure~\ref
{uvfigure}. Here two nodes $u,w$ in different jigsaw clusters $U, W$
know each other ($uw \in E_{\mathrm{people}}$), but their puzzle
pieces are incompatible ($uw \notin E_{\mathrm{puzzle}}$). However,
$u$ and $w$ have merged their puzzle pieces with those of $u'$ and
$w'$, respectively, and $u'$ and $w'$ do have compatible puzzle pieces
($u'w' \in E_{\mathrm{puzzle}}$). Thus $u$ can tell $w$ about her
friend $u'$, and $w$ can tell $u$ about his friend $w'$. Then $u'$ and
$w'$ merge their compatible puzzle pieces, and the jigsaw clusters $U$
and $W$ merge.

%
\begin{figure}

\includegraphics{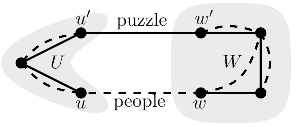}

\caption{Illustration of the jigsaw dynamic. Dashed and solid edges
denote the people graph and puzzle graph, respectively. Jigsaw clusters
$U$ and $W$ contain three and four nodes each. Nodes $u,w$ know each
other but do not have compatible puzzle pieces. However, they have
merged their puzzle pieces with nodes $u', w'$, who do have compatible
puzzle pieces. Thus $U$ and $W$ merge.}\label{uvfigure}
\end{figure}

Our main results, Theorems~\ref{teo-ring} and~\ref{teo-d>1},
characterize a phase transition in the probability that a random graph
solves a jigsaw puzzle in the manner described above. We find, roughly
speaking, the required number of interactions among a group of people
for them to collectively solve a large puzzle. This phase transition
might begin to inform what properties of social networks facilitate
their ability to collaboratively solve problems and to innovate.

\subsection{Related literature}
Previous models of scientific discovery and innovation can be roughly
partitioned into three sets.
Models in the first set focus on the structure of the social network
but not on the space of ideas; an example is an epidemic model of a
single idea that spreads like a slow, hard-to-catch disease in a social
network~\cite{Bettencourt2006,Bettencourt2008}.
Models in the second set focus on the space of ideas but not on the
social network; an example is a branching process of new ideas mating
with old ones~\cite{Sood2010}.
Models in the third set attempt to capture both the social network and
how it interacts with some space of ideas. One example is a model of
people trading and gifting ideas with neighbors in a social network to
obtain certain ideas needed to produce an output~\cite{Cowan2007}.
Four other models in this set are reviewed in~\cite{Chen2009}: an ant
colony model of scientists seeking papers to cite like ants seeking
food; the costs and benefits of hunting for references in bibliographic
habitats (``information foraging theory''); the A--B--C model of
finding triadic closure among ideas; and bridging structural holes
(gaps between dense communities of graphs) in networks of people and
ideas. However, researchers have noted the difficulty in modeling how
teamwork and collaboration lead to greater collective creativity and
discovery~\cite{Bettencourt2009,Duch2010}. Our contribution to this
literature is a model that focuses on the way people might
collaboratively merge their partial solutions to a difficult problem
(or their partial ideas that combine to form a better idea).


\subsection{Road map for the paper}
In Section~\ref{secformaldef}, we define the jigsaw percolation
process formally. We present the main results in Section~\ref
{secresults} and prove them in Sections~\ref{secERproofs}--\ref{power-law}. In Section~\ref{discussion}, we discuss simulations and
open questions.

\section{Formal definition of jigsaw percolation}\label{secformaldef}
Formally, jigsaw percolation on $(V$, $E_{\mathrm{people}}$,
$E_{\mathrm{puzzle}})$ proceeds in steps as follows. At every step
$i\ge0$, we have a partition $\mathcal{C}_{i}$ of the vertex set $V$.
The elements of $\mathcal{C}_i$, called ``jigsaw clusters,'' are
labels on vertices that denote which puzzle pieces have merged by step $i$:
\begin{longlist}[(1)]
\item[(1)] Initially, $\mathcal{C}_0$ is the set of singletons $\{\{
v\}\dvtx  v \in V\}$.
%
\item[(2)] 
At step $(i+1) \geq1$, we merge every pair of jigsaw clusters in
$\mathcal{C}_i$ that are both puzzle- and people-adjacent; see
Figure~\ref{mergeclusters}.
\end{longlist}
For example, after the first step, $\mathcal{C}_1$ is the set of
connected components in the graph $(V,E_{\mathrm{people}}\cap
E_{\mathrm{puzzle}})$.
Note that three or more jigsaw clusters can merge simultaneously, as
illustrated in Figure~\ref{mergeclusters}.

%
\begin{figure}[b]

\includegraphics{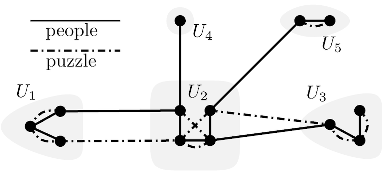}

\caption{Jigsaw clusters $U_1$, $U_2$, $U_3$, $U_4$, $U_5 \in\mathcal{C}_i$
at stage $i$. At stage $i+1$, jigsaw clusters $U_1$, $U_2$, $U_3$ merge.}\label{mergeclusters}
\end{figure}

It is useful to write jigsaw percolation as a dynamical system as
follows. At step $i$, let $\mathcal{E}_i$ be the unordered pairs of
clusters in $\mathcal{C}_i$ that are people-adjacent and
puzzle-adjacent. Then the jigsaw clusters in $\mathcal{C}_{i+1}$ are
the connected components of the graph $(\mathcal{C}_i, \mathcal{E}_i)$:
%
\begin{equation}
\label{comps-iter} \mathcal{C}_{i+1} = \biggl\{\bigcup
_{U \in A} U\dvtx  A \mbox{ is a connected component of } (
\mathcal{C}_i, \mathcal{E}_i) \biggr\}.
\end{equation}

Given $(V, E_{\mathrm{people}}, E_{\mathrm{puzzle}})$, we merge
jigsaw clusters until no more merges can be made, that is, iterate
equation~(\ref{comps-iter}) to a fixed point $\mathcal{C}_\infty$.
After finitely many steps, no more merges can be made. We say that
\emph{the people graph solves the puzzle} if all nodes belong to the
same jigsaw cluster at the end of the process (i.e., \mbox{$\mathcal
{C}_\infty= \{V\}$}). Figure~\ref{dynamicillustration} illustrates a
people graph that fails to solve a $2\times2$ puzzle.

%
\begin{figure}

\includegraphics{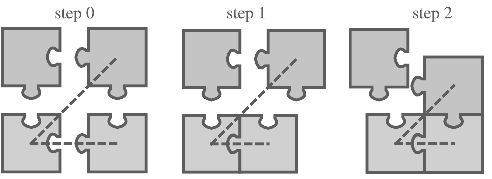}

\caption{A complete trajectory of the jigsaw dynamics. The people
graph (dashed edges) does not solve this $2 \times2$ puzzle.}\label{dynamicillustration}
\end{figure}

An equivalent definition of the process that is elegant and simple to
code on the computer is to iteratively contract nodes that are adjacent
in $E_{\mathrm{people}}\cap E_{\mathrm{puzzle}}$ until no more
contractions are possible. The people graph solves the puzzle if this
procedure ends with a single node.

\section{Statement of results}\label{secresults}
\subsection{Erd\H{o}s--R\'{e}nyi random graphs solving ring and bounded-degree puzzles}
In most of this paper, we consider people graphs that are Erd\H
{o}s--R\'{e}nyi random graphs $\mathcal{G}(n,p_n)$, in which each
possible edge appears independently with probability $p_n$, with
associated probability distribution $\mathbb{P}_{p_n}$. (The exception
is Section~\ref{power-law}, in which we consider power-law random
graphs rather than Erd\H{o}s--R\'{e}nyi random graphs.) For a fixed,
connected puzzle graph of size $n$, we are interested in the
probability of the event
\[
\mbox{\texttt{Solve}}:= \{\mbox{the people graph solves the puzzle}\} = \bigl\{
\mathcal{C}_\infty= \{V\}\bigr\}.
\]
We denote this probability by $\mathbb{P} (\mbox{\texttt{Solve}} )$ or by $\mathbb
{P}_{p_n} (\mbox{\texttt{Solve}} )$ to make explicit the
value of $p_n$. Note that the jigsaw dynamic is \emph{monotonic}, in
that adding more edges to the people graph or to the puzzle graph
cannot decrease the chance of solving the puzzle. Thus, for fixed $n$,
$\mathbb{P}_{p} (\mbox{\texttt{Solve}} )$ is
nondecreasing with $p$. Trivially, $\mathbb{P}_{0} (\mbox
{\texttt{Solve}} )=0$ and $\mathbb{P}_{1} (\mbox{\texttt
{Solve}} )=1$. Furthermore, $\mathbb{P}_{p} (\mbox{\texttt
{Solve}} )$ is a polynomial in $p$ of degree at most ${n\choose
2}$. Thus for each $n$ there exists a unique $p \in(0,1)$ such that
$\mathbb{P}_{p} (\mbox{\texttt{Solve}} ) = 1/2$, and we
make the following definition.

%
\begin{definition}
The \emph{critical value} $p_c(n)$ for solving a connected puzzle is the
unique value of $p_n\in(0,1)$ such that $\mathbb{P}_{p_n} (\mbox{\texttt{Solve}}) = 1/2$.
\end{definition}

%
\begin{remark}
There is nothing special about the number $1/2$. For our results, we
could have taken any fixed positive real number strictly smaller than
$1$. However, the critical value $p_{c}(n)$ depends on the choice of
the puzzle graph, which we suppress in the notation $p_{c}(n)$.
\end{remark}

%
\begin{remark}\label{rmk-connected-people}
If the people graph is not connected, then the puzzle cannot be solved.
Thus $p_{c}(n)\ge t_{n}$, where $t_{n}$ is the unique real number such
that $\mathbb{P} (\mathcal{G}(n,t_n)\mbox{ is connected}
)=1/2$. Asymptotically we have $t_{n}\approx(\log n -\log\log
2)/{n}$; see~\cite{ER61}. Note that the equality $p_{c}(n)=t_{n}$
holds when the puzzle graph is the star graph $(\{1,2,\ldots,n\},\{
(i,n)\dvtx 1\le i<n\})$, because in this case the puzzle can be solved if
and only if the people graph is connected.
\end{remark}

We use the following standard notation for describing sequences of
nonnegative real numbers $a_n$ and $b_n$:
$a_n = O(b_n)$ means there exists $C>0$ so that $a_n \leq Cb_n$ for all
sufficiently large $n$; $a_n = \Theta(b_n)$ means $a_n = O(b_n)$ and
$b_n = O(a_n)$; $a_n = o(b_n)$ means $a_n/b_n \to0$ as $n\to\infty$;
and $a_n = \omega(b_n)$ means $b_n = o(a_n)$.

Our main results are the following two theorems.

%
\begin{theorem}[(Ring puzzle)]\label{teo-ring}
If the people graph is the Erd\H{o}s--R\'{e}nyi random graph and the
puzzle graph is the $n$-cycle, then
\[
\frac{1}{27 \log n}\leq p_c(n) \leq\frac{\pi^2}{6 \log n}\bigl(1+o(1)
\bigr).
\]
Moreover, for $p_{n}=\lambda/\log n$, $\mathbb{P}_{p_n} (\mbox{\texttt{Solve}}) \to0$ or $1$ according as $\lambda
<1/27$ or $\lambda>\pi^2/6$.
\end{theorem}

%
\begin{remark}
We believe that our upper bound is tight; see Section~\ref{discussion}. We did not attempt to optimize the constant $1/27$ in the
lower bound; this value was chosen to make the proof easier to read. We
do not think that our proof method will yield an optimal lower bound.
\end{remark}

%
\begin{theorem}[(Connected puzzle of bounded degree)] \label{teo-d>1}
For an Erd\H{o}s--R\'{e}nyi people graph solving a connected puzzle
with bounded maximum degree,  $p_c(n) = O(1/\log n)$ and $p_c(n) =
\omega(1/n^b)$ for any~$b>0$. In particular, we have $\mathbb
{P}_{p_n} (\emph{\mbox{\texttt{Solve}}} ) \to0$ for
$p_n= O(1/n^b)$ for any $b>0$, and $\mathbb{P}_{p_n} (\emph
{\mbox{\texttt{Solve}}} ) \to1$ for $p_n=\lambda/\log n$ with
$\lambda>\pi^2/6$.
\end{theorem}

%
\begin{remark}
The upper bound for $p_{c}(n)$ in Theorem~\ref{teo-d>1} holds for any
connected puzzle graph, even with maximum degree growing with $n$ as
$n\to\infty$; see Proposition~\ref{ubd-ring}. The star graph example
in Remark~\ref{rmk-connected-people} provides a counterexample to the
lower bound when the maximum degree is unbounded.
\end{remark}

%
\begin{remark}
The jigsaw dynamic is symmetric under swapping the people and puzzle
graphs. Thus Theorems~\ref{teo-ring} and~\ref{teo-d>1} also apply to
a ring and bounded-degree people graph (resp.) solving an Erd\H
{o}s--R\'{e}nyi puzzle.
\end{remark}

Some of the techniques in our proofs resemble those used for long-range
percolation and for bootstrap percolation, but our arguments
differ in key ways. In our proof of the lower bound on $p_c(n)$ for
the ring puzzle graph, we show that a set of cut points, which must
separate jigsaw clusters in the final configuration~$\mathcal{C}_\infty$, exists with high probability for sufficiently
small $p$. This is similar in spirit to finding a positive density
of points over which no edge crosses in the context of
one-dimensional long range
percolation~\cite{Schulman1983,Coppersmith2002} to show that no
infinite component exists.

In our proof of the upper bound on $p_c(n)$, we use the fact that
once a sufficiently large, solved cluster emerges, that cluster will
inevitably continue to merge and ultimately solve the puzzle. As in
bootstrap percolation on the lattice graph~\mbox{\cite{AL88,Holroyd2003}},
our upper bound arises from a sufficient condition for the formation
of a large cluster.

\subsection{Power-law random graphs solving bounded-degree puzzles}
As a model of social networks, the Erd\H{o}s--R\'{e}nyi random graph
assumes no
structure other than the average number of connections
(neighbors) per person. However, in many social networks---from
scientific citations~\cite{Redner1998} to scientific
collaborations~\cite{Barabasi2002,Newman2001phys,Newman2001pnas} to
sexual partners~\cite{Liljeros2001}---some
people have orders of magnitude more connections than others. The
broad-scale degree distributions of such networks are well described
by a power-law (or by a power-law with a cutoff), in which the
fraction of vertices having degree $k$ is proportional to
$k^{-\alpha}$ for some power $\alpha>2$. In light of these findings,
we consider jigsaw percolation on people graphs that are given by
the configuration model~\cite{Molloy1995} with limiting power-law
degree distribution $\mathbf p=\{\degDist_k\}$ satisfying
%
\begin{eqnarray}
\label{distdef} \degDist_k &=& 0\qquad\mbox{for }k< d_{\min}
\mbox{ for some }d_{\min} \ge3\quad\mbox{and}
\nonumber\\[-8pt]\\[-8pt]\nonumber
\degDist_k &\asymp& k^{-\alpha+o(1)}\qquad\mbox{as }k \to\infty
\mbox{ for some power } \alpha>2. 
\end{eqnarray}
The condition $d_{\min}\ge3$ is imposed to ensure that the
resulting people graph is connected with high probability. Here and
later the phrase ``with high probability'' refers to ``with
probability tending to 1 as the size of the graph (network) grows to
infinity.''


In the configuration model, the people graph $(V,E_{\mathrm{people}})$
is constructed in two stages.
Assuming $|V|=n$, first the degrees
$d_1, d_2, \ldots, d_n$ are chosen to be i.i.d. from the aimed
degree distribution $\mathbf p$, and $d_i$ many half-edges are
assigned to vertex $i, 1\le i\le n$. We make the sum of the degrees
even by possibly adding one to~$d_n$. This has no effect on the
analysis that follows. Then, conditioned on $\{d_i\}_{i=1}^n$,
$(V,E_{\mathrm{people}})$ is chosen uniformly from the collection of
(multi-)graphs having degree sequence $(d_1, d_2, \ldots, d_n)$ by
randomly matching the half-edges at each vertex.

Surprisingly, such heterogeneous social networks cannot solve a
large class of puzzles.

%
\begin{proposition} \label{jigsawwithCM}
For any $\alpha>2$, if $(V,E_{\mathrm{people}})$ is given by the configuration
model on $n$ vertices
with power-law degree distribution $\mathbf p$ satisfying (\ref
{distdef}), and if $(V,E_{\mathrm{puzzle}})$ has bounded maximum
degree, then
$\mathbb{P} (\emph{\mbox{\texttt{Solve}}} ) \to0$ as
$n\to\infty$.
\end{proposition}

%
\begin{remark}
Because collaboration networks in
science~\cite{Barabasi2002,Newman2001phys,Newman2001pnas}
manage to collectively solve puzzles despite their degree
distributions being well modeled by power-laws with exponential
decay, more realistic assumptions, such as unbounded-degree puzzles
and randomly grown collaboration networks, merit future work; see
Section~\ref{discussion} for more details.
\end{remark}

For degree exponent $\alpha> 2$ of the social network, we expect
Proposition~\ref{jigsawwithCM} to hold for models of power-law
random graphs other than the configuration model as well. It is easy
to check that the maximum of $n$ i.i.d.~random variables from the
distribution given in~(\ref{distdef}) is tight under the scaling
$n^{-1/(\alpha-1)}$. Thus one expects to couple the power-law
random graph as a subgraph of an Erd\H{o}s--R\`enyi random graph with
edge probability $1/n^{b}$ with $b < 1/(\alpha-1)$ and deduce
Proposition \ref{jigsawwithCM} from Theorem~\ref{teo-d>1} and a
monotonicity argument. This conclusion is indeed true for the
Chung--Lu power-law random graph model (cf.~\cite{cl02}) with
$\alpha>3$. However, for $\alpha<3$ the power-law random graphs
contain large cliques having size polynomial in $n$. This excludes
the possibility of the above coupling, as the maximum size of a
clique in the Erd\H{o}s--R\'{e}nyi random graph $\mathcal
{G}(n,n^{-b})$ is at most
poly-logarithmic in $n$.

The proof of Proposition~\ref{jigsawwithCM}, presented in
Section~\ref{power-law}, circumvents this issue
with a direct argument without the need for any coupling.
Furthermore, for $\alpha\in(1,2)$, we expect the power-law random
graph given by the configuration model to solve any bounded-degree
puzzle with high probability, because then the people graph has very
small diameter; cf.~\cite{infmean}. However, we do not have a
rigorous proof for that conjecture.

\subsection{Subsequent work}
After this work appeared as a preprint, Slivken~\cite{slivken}
proved a related result for random puzzle graph. In this model, both
the people and the puzzle graphs are Erd\H{o}s--R\'{e}nyi with edge
probabilities
$p_{\mathrm{ppl}}$ and $p_{\mathrm{puz}}$, respectively, which satisfy
$p_{\mathrm{ppl}}\wedge p_{\mathrm{puz}} \geq(1+\varepsilon) \log n / n$
for some $\varepsilon>0$ to ensure that both graphs are connected with
high probability. It is shown in~\cite{slivken} that the
probability of solving the puzzle is close to zero if
$p_{\mathrm{ppl}}\cdot p_{\mathrm{puz}} \leq c/(n\log n)$ and is close
to one if $ p_{\mathrm{ppl}} \cdot p_{\mathrm{puz}} \geq\log\log n /(c
n \log n)$, for some constant $c > 0$. In another subsequent
paper~\cite{GS13}, Gravner and one of the present authors proved
that for an Erd\H{o}s--R\'{e}nyi people graph solving a general puzzle
graph with
bounded maximum degree $D$, the critical value $p_{c}$ is
$\Theta(1/\log n)$, where the constants depend only on~$D$.

\section{Erd\H{o}s--R\'{e}nyi random graphs solving ring and bounded-degree puzzles}
\label{secERproofs}
In this section, we prove Theorems~\ref{teo-ring} and \ref{teo-d>1},
in which the people graph is the Erd\H{o}s--R\'{e}nyi random graph.
In Section~\ref{secringupperbound}, we prove the upper bound on the
critical value $p_c(n)$ for both Theorems~\ref{teo-ring} and~\ref
{teo-d>1}. In Section~\ref{secringlowerbound}, we prove the lower
bound for the ring puzzle in Theorem~\ref{teo-ring}, and in
Section~\ref{secd>1lowerbound} we prove the lower bound for
arbitrary puzzles with bounded maximum degree.

\subsection{Upper bound on the critical value}\label{secringupperbound}
In this section, we prove that the critical value has upper bound $\pi
^2/(6 \log n)$ for any connected puzzle graph.

%
\begin{proposition}[(Upper bound for the critical value)]
\label{ubd-ring}
For an Erd\H{o}s--R\'{e}nyi people graph and any connected puzzle
graph on $n$ vertices, if $\lambda> \pi^2/6$ and $p_n = {\lambda
}/{\log n}$, then
\[
\lim_{n\to\infty} \mathbb{P}_{p_n} (\mbox{\texttt{Solve}} )
=~1.
\]
\end{proposition}

%
\begin{remark}
A close look\vspace*{1pt} at the proof of Proposition~\ref{ubd-ring} reveals that
the same conclusion is true as long as $p_{n}\ge\pi^2/(6 \log n)\cdot
(1+{c\log\log n}/{\log n})$ for some constant $c\in(0,\infty)$.
\end{remark}

For simplicity, one can look at the ring puzzle graph (the $n$-cycle), with
\[
E_{\mathrm{puzzle}}= \bigl\{(1,2), (2,3),\ldots,(n-1, n), (n,1)\bigr\}.
\]
The idea of the proof is the following sufficient condition to solve
the ring puzzle, illustrated in Figure~\ref{figupperbound}. Suppose
that in the people graph, node $2$ is adjacent to node~$1$; node $3$ is
adjacent to $1$ or $2$; node $4$ is adjacent to $1$, $2$ or $3$; and so
on, so that node $j$ is people-adjacent to at least one of $\{1, 2,\ldots, j-1\}$ for all $2 \leq j \leq n$ (as illustrated in Figure~\ref
{figupperbound}). Then the people graph solves the puzzle.

%
\begin{figure}[t]

\includegraphics{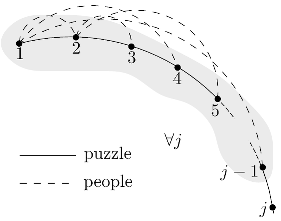}

\caption{Illustration of the sufficient condition to solve the ring
puzzle: $j$ is people-adjacent to $\{1, 2,\ldots, j-1\}$ for all $j=2,
3,\ldots, n$. This event is contained in the event $\mbox{\texttt{Solve}}$.}\label{figupperbound}
\end{figure}

However, to obtain a good bound, we do not consider solving the whole
puzzle in the manner depicted in Figure~\ref{figupperbound}. Instead,
we partition the puzzle graph into disjoint blocks and use the
sufficient condition depicted in Figure~\ref{figupperbound} within
each block. If the blocks are sufficiently large, then solving just one
block suffices to solve the whole puzzle. We call a set $U\subseteq V$
\emph{internally solved} if the people graph induced on $U$ can solve
the puzzle graph induced on $U$ and prove the existence of a ``large''
internally solved set.
We use the following lemma to partition the puzzle graph into disjoint
blocks. The motivation comes from analyzing the ring puzzle graph.

%
\begin{lemma}\label{lemblocks}
Let $m\ge1$ be a fixed integer. For any connected graph $G$ with
vertex set $V$, there exists an integer $k\ge\llvert V\rrvert /(2m)$
and subsets $\mathcal{B}_1,\mathcal{B}_2,\ldots,\mathcal{B}_k$ of
$V$ such that:
\begin{longlist}[(iii)]
\item[(i)] $V=\bigcup_{i=1}^{k} \mathcal{B}_i$;
\item[(ii)] $|\mathcal{B}_i|\in[m,2m]$ for $i=1,2,\ldots,k-1$ and
$|\mathcal{B}_k|<2m$;
\item[(iii)] the induced subgraph on $\mathcal{B}_i$ is connected for
all $i=1,2,\ldots,k$;
\item[(iv)] $\mathcal{B}_i$ and $\mathcal{B}_j$ share at most one
vertex in common for all $1\le i<j\le k$.
\end{longlist}
\end{lemma}

\begin{pf}
The proof proceeds by induction on $n:=|V|$. The lemma is obviously
true for $n \le2m$, so let us assume that $n\ge2m+1$.

For any connected graph $G$ of size $n$, fix a spanning tree $T$ of
$G$. Removing a single vertex $v_0$ from the tree $T$ results in
finitely many disjoint components $C_1,C_2,\ldots,C_k$, each of which
has a unique marked vertex adjacent to $v_0$ in $T$. We consider three
disjoint cases.

\begin{case}\label{case1}
If one of the components has size between $[m,2m]$,
we define this component as $\mathcal{B}_1$ and use induction on the
graph $G$ with the vertex set $\mathcal{B}_1$ removed, which is still
connected.
\end{case}

%
\begin{case}\label{case2}
If all of the components have size $<m$, define $l$ as the smallest
integer such that $|C_1|+|C_2|+\cdots+|C_{l-1}|<m$ and
$|C_1|+|C_2|+\cdots+|C_l|\ge m$. Such an $l$ exists, because
$|C_1|+|C_2|+\cdots+|C_k|=n-1>m$. Necessarily we have
$|C_1|+|C_2|+\cdots+|C_l|< 2m$, because $|C_{i}|<m$ for all $i$. We
take $\mathcal{B}_{1}:=\bigcup_{i=1}^l C_i\cup\{v_0\}$ and use
induction on the graph $G$ with vertex set $\bigcup_{i=1}^l C_i$ removed
(note that $v_0$ will appear in more than one subset because it has not
yet been removed from~$G$).
\end{case}

%
\begin{case}\label{case3}
If none of the components has size between $[m,2m]$ and at least one
component has size $>2m$, we choose one such component (and ignore the
other components), call it $V_1$, and remove the marked vertex $v_1$
from it. Removing $v_1$ creates several new components, each containing
a marked vertex adjacent to $v_{1}$ in $T$. We repeat this procedure
until reaching the following situation: the size of $V_{k}$ is $> 2m$,
but if we remove the marked vertex $v_k$ from it, then all the
resulting components have size $\le2m$. If one of them has size more
than $m$, then we take that component as $\mathcal{B}_1$, and we
continue by induction with the rest of the tree, which is connected by
construction. If all of the components have size $< m$, we follow the
steps in Case~\ref{case2} to define $\mathcal{B}_1$ and continue by induction.
\end{case}

To complete the proof we need to check properties (iii) and (iv) for
each block~$\mathcal{B}_i$, which follow easily from the spanning tree
and marked vertex construction.
\end{pf}

\begin{pf*}{Proof of Proposition~\ref{ubd-ring}}
Using Lemma~\ref{lemblocks}, we partition the puzzle graph into
blocks $\mathcal{B}_{1}$, $\mathcal{B}_{2}, \ldots, \mathcal
{B}_{k}$ of size $\le2m$ (where $m$ is determined later) with
$|\mathcal{B}_{i}|\ge m$ for all $i<k$. Note that $k\ge n/(2m)$. Let
$B_i$ be the event that block $\mathcal{B}_i$ is solved using only
people edges in block $\mathcal{B}_i$. Let $S:= \sum_{i=1}^{k-1}
\mathbh{1}_{B_i}$ be the number of blocks (excluding the last block
$\mathcal{B}_{k}$) that are solved using people edges only within each
block (i.e., internally solved). The events $B_i$ are independent
because the blocks use disjoint sets of edges, and they are Bernoulli
random variables with mean~$\mathbb{P} (B_i )$.

Next we show that if $p_n = \lambda/ \log n$ with $\lambda> \pi^2 /
6$, then
\[
\mathbb{P} (S\ge1 )\to1\qquad\mbox{as }n \to\infty.
\]

Consider the subgraph of the puzzle graph induced by $\mathcal
{B}_{i}$. We can fix a rooted spanning tree and label the vertices with
integers $1,2,\ldots,|\mathcal{B}_{i}|$ in such a way that the vertex
with label $j$ is puzzle-adjacent to the set of vertices with labels $\{
1,2,\ldots,j-1\}$ in the spanning tree for all $j\ge1$. As
illustrated in Figure~\ref{figupperbound}, a~sufficient condition for
the event $B_i$ to occur is the event
\begin{eqnarray*}
\overline{B_i}&:= &\bigl\{\mbox{for all }1\le j\le|
\mathcal{B}_i|,\mbox{ the vertex labeled $j$ is
people-adjacent}
\\
&&\hspace*{68pt}\mbox{to the set of vertices labeled $\{1, 2,\ldots, j-1\} $}\bigr\}
\subset B_i.
\end{eqnarray*}
[Note that there could be other ways to solve the puzzle. For example,
in the case of a ring puzzle, $j$ is people-adjacent to $j+1$, and
$j+1$ (but not $j$) is people-adjacent to $\{1,\ldots, j-1\}$. Thus
$\overline{B_1}$ is not a necessary condition for $B_1$ to occur, that
is, $\overline{B_1} \subsetneq B_1$.] The events that $j+1$ is
people-adjacent to $\{1, 2,\ldots, j\}$ occur independently with
probability $\ge1-(1-p_n)^{j}$, so
\[
\mathbb{P} (\overline{B_i} ) \ge\prod
_{j=1}^{|\mathcal
{B}_{i}|-1} \bigl( 1- (1-p_n)^{j}
\bigr) \ge\prod_{j=1}^{2m} \bigl( 1-
(1-p_n)^{j} \bigr).
\]
Thus the random variable $S$ stochastically dominates
\[
S' \sim\operatorname{Binomial}\Biggl(k-1,\prod
_{j=1}^{2m} \bigl( 1- (1-p_n)^{j}
\bigr)\Biggr).
\]

For $n \in\mathbb{N}$, let $\varepsilon_n:= - \log(1-p_n)$, so that
$1-p_n=e^{-\varepsilon_n}$. We use the next lemma to obtain a lower bound on
\begin{eqnarray*}
\log\mathbb{E} S' &=& \log(k-1) + \sum
_{j=1}^{2m} \log \bigl( 1- e^{-j \varepsilon_n} \bigr).
\end{eqnarray*}
The proof of Lemma~\ref{numest} follows the present proof.

%
\begin{lemma}\label{numest}
Let $\theta(x):=-\int_0^x\log(1-e^{-t}) \,dt$ for $x\in[0,\infty
]$. If $\lim_{\varepsilon\to0}m_{\varepsilon}\varepsilon=x\in[0,\infty]$, then
\[
\lim_{\varepsilon\to0} \varepsilon\sum_{i=1}^{m_{\varepsilon}}
\log \bigl(1-e^{-i\varepsilon}\bigr) = -\theta(x).
\]
Moreover, for all $m \geq1$ and $\varepsilon> 0$,
%
\begin{equation}
\label{eqlogpr-bd} \Biggl\llvert \sum_{i=1}^{m}
\log\bigl(1-e^{-i\varepsilon}\bigr) + \frac{\pi^2}{6
\varepsilon}\Biggr\rrvert \leq
\frac{1}{2} \log\frac{2e^2}{\varepsilon} + \frac{\pi^2}{6 \varepsilon e^{m\varepsilon}}.
\end{equation}
\end{lemma}

Fix $\delta>0$, and let $m:=\lceil(1+\delta) (\log n)/\varepsilon_n
\rceil$. Here we tacitly assume that $n$ is large, so that $2m< n$.
Using Lemma~\ref{numest}, we estimate
\begin{eqnarray*}
\log\mathbb{E} \bigl( S'\bigr) &\ge&\log \biggl(\frac{n}{2m}-1
\biggr) - \frac{\pi^2}{6\varepsilon_n} + \Biggl( \sum_{j=1}^{2m}
\log\bigl(1-e^{-j\varepsilon_n}\bigr) + \frac{\pi
^2}{6\varepsilon_n} \Biggr)
\\
&\geq& \biggl(1-\frac{\pi^2}{6\lambda} \biggr)\log n - \log\frac
{2m}{1-2m/n} -
\frac{1}{2} \log\frac{2e^2}{\varepsilon_n} - \frac{\pi
^2}{6 \varepsilon_n e^{2m\varepsilon_n}}
\\
&\ge& \biggl(1-\frac{\pi^2}{6\lambda} \biggr)\log n -\log\frac
{m}{\sqrt{\varepsilon_n}}- O(1)
\\
&\ge& \biggl(1-\frac{\pi^2}{6\lambda} \biggr)\log n -\frac{5}{2}\log \log n -
O(1)
\\
&\to&\infty\qquad\mbox{as } n \to\infty.
\end{eqnarray*}
In the last inequality we used the fact that $m=O(\log n/\varepsilon_n)$
and $\varepsilon_n\ge p_n =\lambda/\log n$. Since $S'$ is binomial,
$\mathbb{E} (S') \to\infty$ implies that $\mathbb{P} (S'\geq1 ) \to1$.

Let $I:=\inf\{i\ge1\dvtx  \mathcal{B}_{i}$ is internally solved$\}$ be
the random index such that $\mathcal{B}_{I}$ is the first block among
$\mathcal{B}_{1},\mathcal{B}_{2},\ldots$ that is internally solved.
We define $I=\infty$ when no internally solved block exists. Thus we
have $\mathbb{P} (I<\infty )=\mathbb{P} (S\ge1 )\ge\mathbb{P} (S'\ge1 )\to1$ as $n\to\infty$.

Let $U$ be a deterministic set of size $m$. The probability that all
the remaining $n-m$ vertices in $V\setminus U$ are connected to $U$ by
a people edge is
\[
\bigl(1-(1-p_n)^m\bigr)^{n-m} \ge
\bigl(1-e^{-\varepsilon_n m}\bigr)^n\ge1-ne^{-\varepsilon_n
m}
\ge1-n^{-\delta}.
\]
Note that by connectivity of the puzzle graph and people graph, the
event that all vertices in $V\setminus U$ are connected to $U$ by
people edges and $U$ is internally solved implies \mbox{\texttt
{Solve}}. Moreover the event that a particular set of vertices forms an
\emph{internally solved} subset or not depends only on the edges among
those vertices.
Thus we have
\begin{eqnarray*}
\mathbb{P} (\mbox{\texttt{Solve}} )&\ge& \mathbb{P} (\mbox{\texttt{Solve}}, I<\infty )
\\
&\ge&\sum_{i=1}^k\mathbb{P} (\mbox{\texttt{Solve}}|I=i )\mathbb{P} (I=i )\ge
\bigl(1-n^{-\delta}\bigr)\mathbb{P} (I<\infty )\to1
\end{eqnarray*}
as $n\to\infty$. The proof is complete.
\end{pf*}

\begin{pf*}{Proof of Lemma~\ref{numest}}
Note that
\begin{eqnarray*}
-\varepsilon\sum_{i=1}^{k}\log
\bigl(1-e^{-i\varepsilon}\bigr)
&=& \varepsilon\sum_{i=1}^{k}
\sum_{j=1}^{\infty}\frac{e^{-ij\varepsilon}}{j} = \varepsilon
\sum_{j=1}^{\infty}\frac{1-e^{-jk\varepsilon
}}{j(e^{j\varepsilon}-1)}
\\
&=& \sum_{j=1}^{\infty}\frac{1-e^{-jk\varepsilon}}{j^2} -
\sum_{j=1}^{\infty} \frac{(1-e^{-jk\varepsilon})(e^{j\varepsilon}-1-j\varepsilon
)}{j^2(e^{j\varepsilon}-1)}.
\end{eqnarray*}
Using the power series expression of $e^x$, it is easy to see that
$(e^x-1-x)/(e^x-1)\leq\min\{x/2,1\}$. Applying the last inequality,
we have
\begin{eqnarray*}
\sum_{j=1}^{\infty} \frac{(1-e^{-jk\varepsilon})(e^{j\varepsilon
}-1-j\varepsilon)}{j^2(e^{j\varepsilon}-1)}
&\leq&
\sum_{j=1}^{\infty} \frac{\min\{j\varepsilon/2,1\}}{j^2}
\leq\sum_{j\leq m}\frac{\varepsilon}{2j} +\sum
_{j>m}\frac{1}{j^2}
\\
&\leq& \frac{\varepsilon}{2} (\log m + 1) +
\frac{1}{m}
= \frac
{\varepsilon}{2}\log\frac{2e^2}{\varepsilon}
\end{eqnarray*}
using $m=2/\varepsilon$. Thus, combining the last two displays,
%
\begin{equation}
\label{eqlogpr-bd-lem} \Biggl\llvert \sum_{i=1}^{k}
\varepsilon\log\bigl(1-e^{-i\varepsilon}\bigr) + \sum_{j=1}^{\infty}
\frac{1-e^{-jk\varepsilon}}{j^2}\Biggr\rrvert \leq\frac
{\varepsilon}{2}\log\frac{2e^2}{\varepsilon}.
\end{equation}
In particular, if $\lim_{\varepsilon\to0}k_{\varepsilon}\varepsilon=x\in
[0,\infty]$, then interchanging the sum and the integral
\begin{eqnarray*}
\lim_{\varepsilon\to0} \varepsilon\sum_{i=1}^{k_{\varepsilon}}
\log \bigl(1-e^{-i\varepsilon}\bigr)
&=& -\sum_{j=1}^{\infty}
\frac{1-e^{-jx}}{j^2}
\\
&=&-\sum_{j=1}^{\infty}\frac{1}j \int
_0^x e^{-jt} \,dt
= \int
_{0}^{x}\log\bigl(1-e^{-t}\bigr) \,dt,
\end{eqnarray*}
which\vspace*{1pt} completes the proof. The bound~(\ref{eqlogpr-bd}) follows
from~(\ref{eqlogpr-bd-lem}) and the fact that $e^{-jk\varepsilon}\leq
e^{-k\varepsilon}$ for all $j\geq1$.
\end{pf*}

\subsection{Lower bound for the ring puzzle} \label{secringlowerbound}
In this section, we prove a matching-order lower bound for an Erd\H
{o}s--R\'{e}nyi people graph solving the ring puzzle. The idea of the
proof is to show the existence of a cut set that divides the ring into
pieces that never merge.

\begin{proposition}
\label{teoring-lower-bound}
For the ring puzzle graph, if $\lambda\le1/27$ and $p_n = \lambda
/{\log n}$, then $\mathbb{P}_{p_n} (\mbox{\texttt
{Solve}} )\to0$. Therefore $p_c(n) \geq1/(27 \log n)$.
\end{proposition}

\begin{pf}
Let $x$ be a positive integer to be chosen later [it~will be $\Theta
(\log n)$]. We will identify the vertices in the ring puzzle graph
$(V,E_{\mathrm{puzzle}})$ with elements from $\mathbb{Z}_n$, so that
two vertices $u,v\in\mathbb{Z}_n$ are neighbors if and only if
$u-v=\pm1$, where all additions and subtractions in $\mathbb{Z}_n$
are modulo $n$. We denote the interval $\{a,a+1,\ldots,b\}\subseteq
\mathbb{Z}_n$ by $[a,b]$ and its length by $|[a,b]|=b-a+1$.

Given an interval $I=[a,b]\subset\mathbb{Z}_n$, we call it \textit{$x$-good} if there is a vertex $u\in I$
such that $u$ is not people-adjacent to any vertex in the interval
$[a-x,b+x]$. We call the vertex $u\in I$ an \textit{$x$-good vertex in
$I$}. The proof hinges on the following observation.
Loosely speaking, if throughout the puzzle there are people
unacquainted with anyone in a sufficiently large neighborhood of the
puzzle, then these people obstruct the growing solution, and the social
network cannot solve the puzzle.

%
\begin{lemma}\label{lem-x-good}
Suppose that there exist integers $0=a_{0}<a_{1}<\cdots<a_{k}= n$ such
that, for all $j=0,1,\ldots,k-1$, the interval
$I_{j}:=[a_{j}+1,a_{j+1}]$ is $x$-good and has length $|I_j| \leq x$.
Then the puzzle cannot be solved.
\end{lemma}

\begin{pf}
Let $v_{j}\in I_{j}$ be an $x$-good vertex in $I_{j}$ for $j=0,1,\ldots,k-1$. Clearly $1\le v_{0}<v_{1}<\cdots<v_{k-1}\le n$. Furthermore,
each $v_{j}$ has no people edges with $[v_{j-1},v_{j+1}]$ (where
$j+\ell$ is taken modulo $k$) because $|I_j| \leq x$ for all
$j=0,1,\ldots,k-1$.

Suppose for contradiction that the puzzle can be solved. Then there
must exist a first stage, $i$, after which there exists an index $j$
such that two distinct vertices, $u \in[v_j, v_{j+1}]$ and $v\in
[v_{j+1},v_{j+2}]$, belong to the same cluster in $\mathcal{C}_i$. One
of these vertices must be $v_{j+1}$ (without loss of generality,
$u=v_{j+1}$), because otherwise $v_{j+1}$ would have to belong to a
larger cluster in $\mathcal{C}_{i-1}$, and therefore $v_{j+1}$ would
have merged at an earlier stage of the process, which is a
contradiction. Since $v_{j+1}$ is not people-adjacent to any other
vertices in $[v_{j+1},v_{j+2}]$, $v$ must be in a component in
$\mathcal{C}_{i-1}$ that contains vertices outside of
$[v_{j+1},v_{j+2}]$, but this is also a contradiction. Thus the puzzle
cannot be solved.
\end{pf}

In light of Lemma~\ref{lem-x-good}, to complete the proof we need to
show the existence of such intervals with probability tending to $1$.
Suppose $n\ge x^2$. Define $k:= \lfloor n/(x-1)\rfloor\le n$. Define
\begin{eqnarray*}
l_{i}&:=& x\qquad\mbox{for } 1\le i\le n-k(x-1),
\\
l_{i}&:=& x-1\qquad\mbox{for } n-k(x-1)<i\le k,
\end{eqnarray*}
and
$
a_{i}:= l_{1}+l_{2}+\cdots+l_{i}$ for $i=0,1,\ldots,k$.
Note that $a_{k}=n$.
Clearly all the intervals $I_{i}:=[a_{i}+1,a_{i+1}], 0\le i\le k-1$ are
of length $x-1$ or $x$. Let $Z$ be the number of intervals that are not
$x$-good,
\[
Z:= \sum_{i=0}^{k-1} \mathbh{1}_{\{\mathrm{the\ interval}\ I_i
\ \mathrm{is\ NOT\ x\mbox{-}good}\}}.
\]
It suffices to show that
$
\mathbb{P} (Z >0 ) \to0$ as $n\to\infty$
for appropriate choice of $x$. We will use Lemma~\ref{lemlocally-isolated} to estimate the probability that an interval is
not $x$-good.

%
\begin{lemma}\label{lemlocally-isolated}
Fix an integer $x\ge1$. Let $I$ be an interval of length $lx$ for some
number $l>0$. Suppose that $t:=px\in(0,1/(l+2))$. Then we have
\begin{eqnarray*}
\mathbb{P} (I \mbox{ is NOT $x$-good} )
&\le&\exp \biggl[-\frac{t}{2p} \bigl(2l\log(\sqrt {1+{l}/{t}}-1) +
\bigl(l^2+4l+2\bigr)t
\\
&&\hspace*{92pt}{}- 2t\sqrt{1+{l}/{t}}-2l\log l -l \bigr) \biggr].
\end{eqnarray*}
\end{lemma}

In our case, all intervals are of length $x-1$ or $x$, so $l\in
[1-1/x,1]$. If we suppose that $t:=px<1/3$, then
\begin{eqnarray*}
\mathbb{P} (Z >0 ) &\le&\mathbb{E} (Z)
\\
&\le& n \exp \biggl[-
\frac{t}{2p} \bigl(2\log(\sqrt{1+1/t}-1) + 7t - 2t\sqrt{1+1/t} -1 + \eta(x)
\bigr) \biggr],
\end{eqnarray*}
where $\eta(x)\to0$ when $x\to\infty$.
In particular, if $p=p_{n}=\lambda/\log n$ and $x=t\log n/\lambda$
for some $t<1/3$, we have
\begin{eqnarray*}
\mathbb{P} (Z>0 ) &\le&\exp \biggl[\log n -\frac{t\log n}{2\lambda} \bigl(2
\log(\sqrt {1+1/t}-1) + 7t
\\[-2pt]
&&\hspace*{86pt}{} - 2t\sqrt{1+1/t} -1+\eta(t\log n/\lambda) \bigr)
\biggr]
\\[-2pt]
&\to& 0\qquad\mbox{as } n\to\infty
\end{eqnarray*}
when
%
\begin{equation}
\label{eqlineq} \lambda< \frac{t}{2} \bigl[ 2\log(\sqrt{1+1/t}-1) + 7t - 2t
\sqrt {1+1/t} -1 \bigr].
\end{equation}
One can easily check (by taking $t=0.07$) that
\[
\sup_{t\in(0,1/3)} \frac{t}{2} \bigl[ 2\log(\sqrt{1+1/t}-1) +
7t - 2t\sqrt{1+1/t} -1 \bigr] > 1/27.
\]
Thus given $\lambda\le1/27$, we can choose $t\in(0,1/3)$ such that
(\ref{eqlineq}) holds, and taking $x=t\log n/\lambda$ we have
\[
\mathbb{P} \Biggl(\sum_{i=0}^{k-1}
\mathbh{1}_{\{\mathrm{the\ interval}\ I_i\ \mathrm{is\ NOT}\ x\mbox{-}\mathrm{good}\}} >0 \Biggr)\to0\qquad\mbox{as } n\to\infty.
\]
This completes the proof.
\end{pf}

\begin{pf*}{Proof of Lemma~\ref{lemlocally-isolated}}
Without loss of generality, suppose that the interval $I$ is $[1,lx]$.
Recall that $I$ is $x$-good if there is a vertex $u\in I$ such that $u$
has no people edges with $I_{x}:=[1-x,lx+x]$. Thus $I$ is not $x$-good
implies that all vertices in $I$ have at least one people edge with
$I_{x}$, in other words $\sum_{j\in I_x} \mathbh{1}_{\{ i\ \mathrm{has\
a\ people\ edge\ with}\ j\}} \ge1$ for all $i\in I$, and thus
\[
\sum_{i\in I} \sum_{j\in I_x}
\mathbh{1}_{\{ i\ \mathrm{has\ a\ people\ edge\ with}\ j\}}\ge lx.
\]
The number of distinct pairs of vertices between $I$ and
$I_{x}\setminus I$ is $2lx^2$, and the number of distinct pairs of
vertices within $I$ is ${lx\choose2}$. Therefore
\[
\sum_{i\in I} \sum_{j\in I_x}
\mathbh{1}_{\{ i\ \mathrm{has\ a\ people\ edge\ with}\ j\}} \stackrel{\mathrm{d}} {=}X+2Y,
\]
where $X\sim\operatorname{Bin}(2lx^{2},p), Y\sim\operatorname
{Bin}({lx\choose2},p)$ and $X,Y$ are independent. In particular, we have
\begin{eqnarray*}
\mathbb{P} (I \mbox{ is not $x$-good} )&\le&\mathbb{P}
(X+2Y\ge lx )
\\
&\le&\mathbb{P} \bigl(X+2Y'\ge lx \bigr) \le
e^{-\theta lx}\mathbb{E} \bigl(e^{\theta X+2\theta Y'}\bigr)
\end{eqnarray*}
for any $\theta>0$, where $Y'\sim\operatorname{Bin}(l^2x^2/2,p)$ is
independent of $X$. We have
\begin{eqnarray} \label{eqbd1}
\mathbb{P} \bigl(X+2Y'\ge lx \bigr) &\le&
e^{-\theta lx} \bigl(1-p+pe^{\theta}\bigr)^{2lx^2}
\bigl(1-p+pe^{2\theta
}\bigr)^{l^2x^2/2}
\nonumber
\nonumber\\[-8pt]\\[-8pt]
&\le&\exp\bigl[ -lx\bigl(\theta-2t\bigl(e^{\theta}-1\bigr)-lt
\bigl(e^{2\theta}-1\bigr)/2\bigr)\bigr],\nonumber
\end{eqnarray}
where $t:=px$. Note that we have
\[
\frac{\mathbb{E} (X+2Y')}{lx}= (l+2)px=(l+2)t.
\]
Hence, under the assumption $t\in(0,1/(l+2))$, we have $lx>\mathbb
{E} (X+2Y')$ and $\sqrt{1+l/t}-1>l$. Taking $\theta=\log[(\sqrt
{1+l/t}-1)/l]$ in (\ref{eqbd1}), we finally have
\begin{eqnarray*}
\mathbb{P} (I \mbox{ is not $x$-good} )
&\le&\exp \biggl[-\frac{t}{2p} \bigl(2l\log(\sqrt{1+{l}/{t}}-1) +
\bigl(l^2+4l+2\bigr)t
\\
&&\hspace*{92pt}{} - 2t\sqrt{1+{l}/{t}}-2l\log l -l \bigr) \biggr].
\end{eqnarray*}

This completes the proof.
\end{pf*}

Propositions~\ref{ubd-ring} and~\ref{teoring-lower-bound} give
Theorem~\ref{teo-ring}.

\subsection{Lower bound for puzzles with bounded degree} \label{secd>1lowerbound}
In this section, we prove the lower bound in Theorem~\ref{teo-d>1} for
arbitrary puzzle graphs with bounded degree as $n\to\infty$.

%
\begin{proposition}
\label{lbd-d>1}
For any sequence of connected puzzle graphs with boun\-ded maximum degree
as $\llvert V\rrvert =n\to\infty$, $p_c(n) = \omega(1/n^b)$ for any
$b > 0$.
\end{proposition}

\begin{pf}
Let $p = n^{-b}$ such that $k\geq2$ and $b\in(\frac{1}{k},\frac
{1}{k-1})$ are fixed, and suppose that the maximum degree of
$(V,E_{\mathrm{puzzle}})$ is at most $D$ for all $n$. After stage $i$
we have a collection of jigsaw clusters $\mathcal{C}_i$. Initially
$\mathcal{C}_0 = \{\{v\}\dvtx  v \in V\}$, and after the first stage
$\mathcal{C}_1$ is the set of connected components in the graph
$(V,E_{\mathrm{people}}\cap E_{\mathrm{puzzle}})$. Thereafter, two
clusters $U,U'\in\mathcal{C}_i$ merge if there is an edge between the
two clusters in $E_{\mathrm{people}}$ and an edge between the two
clusters in $E_{\mathrm{puzzle}}$. Therefore, if $U, U' \in\mathcal
{C}_i$, then $U, U' \subset W \in\mathcal{C}_{i+1}$ if and only if
there is some nonnegative integer~$\ell$ and a sequence of clusters
$U= U_0, U_1, \ldots, U_{\ell} = U' \in\mathcal{C}_i$ such that
$U_j$ merges with $U_{j+1}$ at stage $i+1$.

Observe that for $i\geq1$, every merge event in stage $i+1$ must
involve at least one cluster that was formed by a merge in stage $i$.
Inspired by this observation, we let $\mathcal{A}_i \subseteq\mathcal
{C}_i$ be the set of \textit{active} clusters that were the result of at
least one merge in stage $i$ when $i\geq1$, and let $\mathcal{A}_0 =
\mathcal{C}_0$. Next we define the events $E_i$ and $F_i$ for $i=0,
\ldots, k$ as
\begin{eqnarray*}
E_i &=& \bigl\{\llvert \mathcal{A}_i\rrvert \geq
C_i n^{1-ib} \bigr\},
\\
F_i &=& \bigl\{ \max\bigl\{\llvert W\rrvert\dvtx  W \in
\mathcal{C}_i\bigr\} \geq L_i \bigr\},
\end{eqnarray*}
where $C_i$ and $L_i$ are constants that depend on $d$ and $k$, which
we will define later. In words, $E_i$ is the event that there are at
least $C_i n^{1-ib}$ active clusters following stage $i$, which is
contained in the event that at least $C_i n^{1-ib}$ merges occur at
stage $i$, because each active cluster must be the result of at least
one merge. $F_i$ is the event that the largest cluster following stage
$i$ has at least $L_i$ vertices. For sufficiently large $n$, the event
$E_k$ is equivalent to the event that at least one merge occurs at
stage $k$, because $kb > 1$. Therefore, our goal is to show that $
\mathbb{P} (E_k ) \to0$ and $\mathbb{P} (F_k ) \to0$ as $n\to\infty$, which implies
that no merges occur after stage $k$ and that the largest cluster has
size at most $L_k$, so the puzzle remains unsolved.

Our strategy is to prove this by induction on $i$. It is trivially true
that $\mathbb{P} (E_0 ) = 0$ and $\mathbb{P} (F_0 ) = 0$ with $C_0 = 2$ and $L_0 = 2$. Now,
let us assume that $\mathbb{P} (E_i )\to0$ and $\mathbb{P} (F_i )\to0$ as $n\to\infty$ for some $i \in\{
0,1, \ldots, k-1\}$, which implies that $\mathbb{P} (E_i^c \cap F_i^c ) \to1$. On the event
$E_i^c\cap F_i^c$, we know that the number of active clusters is $\llvert \mathcal{A}_i\rrvert  < C_i n^{1-ib}$, and the largest cluster has at
most $L_i$ vertices. The latter implies that every cluster has fewer
than $DL_i$ neighboring clusters in $(V,E_{\mathrm{puzzle}})$ because
each vertex has at most $D$ total neighboring vertices in the puzzle
graph. We will use this fact in two ways. First, we will show that the
number of merges at stage \mbox{$i+1$} is small because each active cluster
after stage $i$ has relatively few opportunities to merge. Second, we
will show that no path of neighboring clusters longer than length $k-i$
merge at stage $i+1$ because few such paths exist.

To meet our first goal, we define a random variable $I_{\{A,B\}}^{i+1}$
for\vspace*{1pt} each pair of an active cluster $A\in\mathcal{A}_i$ and a
neighboring cluster $B \in\mathcal{C}_i$ such that $B\neq A$, and
there is an edge in $E_{\mathrm{puzzle}}$ between $A$ and $B$. The
random variable $I_{\{A,B\}}^{i+1}$ is the\vspace*{1pt} indicator of the event that
$A$ and $B$ merge at stage $i+1$. On the event $F_i^c$, the probability
that $A$ merges with $B$ is at most
%
\begin{equation}
\label{merge-prob} 1 - \bigl(1 - n^{-b} \bigr)^{(DL_i)^2} \leq1 -
\bigl(1-(DL_i)^2 n^{-b}\bigr) =
(DL_i)^2 n^{-b},
\end{equation}
where 
we use the fact that $(1-x)^n\ge1-nx$ for $x\in(0,1)$. For
convenience, we now order the clusters in $\mathcal{C}_i$ so that
$A_1, A_2, \ldots, A_{\llvert \mathcal{A}_i\rrvert } \in\mathcal
{A}_i$ and $A_{\llvert \mathcal{A}_i\rrvert  + 1}, A_{\llvert \mathcal
{A}_i\rrvert  + 2}, \ldots, A_{\llvert \mathcal{C}_i\rrvert } \in
\mathcal{C}_i \setminus\mathcal{A}_i$. Therefore, on $E_i^c \cap
F_i^c$, the total number of merges that occur in stage $i+1$,
\[
\sum_{j= 1}^{\llvert \mathcal{A}_i\rrvert } \sum
_{\ell= j+1}^{\llvert \mathcal{C}_i\rrvert } I_{\{A_j, A_{\ell}\}}^{i+1},
\]
is stochastically dominated by $X_i \sim$~Binomial$(DL_i C_i n^{1-ib},
(DL_i)^2 n^{-b})$. This is because there are at most $DL_i C_i
n^{1-ib}$ distinct pairs of neighboring clusters, at least one of which
is active, and the events that each of these pairs merges at stage
$i+1$ are independent because they depend on disjoint sets of edges in
the people graph. If we let $C_{i+1} = 2(DL_i)^3 C_i$ (this is $2
\mathbb{E} X_i / n^{1-(i+1)b}$), then by Chebyshev's inequality
\begin{eqnarray*}
\mathbb{P} \bigl(E_{i+1} | E_i^c
\cap F_i^c \bigr) &=&
\mathbb{P} \Biggl(\sum
_{j= 1}^{\llvert \mathcal{A}_i\rrvert } \sum
_{\ell= j+1}^{\llvert \mathcal{C}_i\rrvert } I_{\{A_j, A_{\ell}\}
}^{i+1} \geq
C_{i+1} n^{1-(i+1)b} \bigg| E_i^c \cap
F_i^c \Biggr)
\\
&\leq&
\mathbb{P} \bigl(X_i \geq C_{i+1}
n^{1-(i+1)b} \bigr)
\\
&=&
\mathbb{P} (X_i - \mathbb{E} X_i \geq
\mathbb{E} X_i )
\\
&\leq& (\mathbb{E} X_i)^{-1} = O\bigl(n^{-1 + (i+1)b}
\bigr) \to0.
\end{eqnarray*}
Since $\mathbb{P} (E_i^c \cap F_i^c ) \to1$, we have that $
\mathbb{P} (E_{i+1} ) \to0$.

Next we must show that the largest cluster after stage $i+1$ has size
at most $L_{i+1}$. Define a \textit{cluster path} of length $\ell\geq0$
between $U,U' \in\mathcal{C}_i$ to be a sequence of distinct clusters
$U = U_0, U_1, \ldots, U_{\ell} = U' \in\mathcal{C}_i$ such that
$U_j$ and $U_{j+1}$ are puzzle-adjacent for all $j \in\{0, \ldots,
\ell-1\}$. For a fixed cluster $A \in\mathcal{C}_i$, let $Y^i_A$
denote the number of cluster paths of length $k$ that start at $A$
(meaning that $U_0=A$) and such that $U_j$ will merge with $U_{j+1}$ at
stage $i+1$ for each $j \in\{0, \ldots, k-1\}$. For any cluster path
$U_0, \ldots, U_k$, the probability that $U_j$ and $U_{j+1}$ merge at
stage $i+1$ is bounded above by $(DL_i)^2 n^{-b}$ on the event $F_i^c$,
by inequality~(\ref{merge-prob}). The number of cluster paths of
length $k$ in after stage $i$ that start at $A$ is bounded by
$(DL_i)^k$ on $F_i^c$ because each cluster has at most $DL_i$
neighboring clusters. Therefore, by Markov's inequality,
\begin{eqnarray*}
\mathbb{P} \biggl(\sum_{A\in\mathcal{C}_i}
Y^i_A \geq1 \Big| F_i^c
\biggr) &\leq& n\mathbb{P} \bigl(Y^i_A \geq1
| F_i^c \bigr)
\\
&\leq& n \bigl[(DL_i)^k \bigl((DL_i)^2
n^{-b} \bigr)^k \bigr] = O\bigl(n^{1-kb}\bigr)
\to0.
\end{eqnarray*}
This implies that there are no cluster paths of length $k$ or longer
that merge at stage $i+1$. Note that clustering can occur in any
tree-like pattern, and the maximum size of a rooted tree with depth
(maximum distance from the root) $k$ and maximum degree $DL_{i}$ is
$L_i (1 + (DL_i)^1 + (DL_i)^2 + \cdots+ (DL_i)^k) = L_i
((DL_i)^{k+1}-1)/(DL_i-1)$.

In\vspace*{1pt} turn, this implies that the largest cluster after stage $i+1$ is
smaller than $L_{i+1}:= L_i ((DL_i)^{k+1}-1)/(DL_i-1)$ with high
probability on the event $F_i^c$, so $\mathbb{P} (F_{i+1} ) \to0$, which completes the proof.
\end{pf}

Propositions~\ref{ubd-ring} and~\ref{lbd-d>1} give Theorem~\ref{teo-d>1}.

\section{People graphs with limiting power-law degree distributions}\label{power-law}
In this section, we prove Proposition~\ref{jigsawwithCM}, which
states that a configuration model random people graph with limiting
power-law degree distribution having exponent $\alpha>2$ \mbox{cannot}
solve bounded-degree puzzles with high probability. Recall that a
set $U\subseteq V$ is \emph{internally solved} if the people graph
induced on $U$ can solve the puzzle graph induced on $U$. We will
call this event $\mbox{\texttt{Solve}}_U$. The idea is to show that
with high
probability no set of vertices of a certain, finite size is
internally solved.

%
\begin{lemma}\label{solvesubsetlemma}
Suppose $U\subseteq V$ such that $\llvert U\rrvert  = m>1 + \frac
{2\alpha}{\alpha- 2}$ is constant. Then
\[
\mathbb{P} (\emph{\mbox{\texttt{Solve}}}_U ) = o
\bigl(n^{-1}\bigr).
\]
\end{lemma}

\begin{pf}
Without loss of generality, suppose that $U=[m]$. Fix $\gamma:=\alpha
/2\in(1,\alpha-1)$ and $\varepsilon:=1/2-1/\alpha$, so that
$(1-\varepsilon)\gamma>1$. It is easy to see that $\mathbb{E}
d_1^\gamma<\infty$. Define the event
\[
\mathcal{D}_{n,m}:=\bigl\{ \mbox{there exists a pair of indices $1\le
i<j\le m$, such that $d_i d_j \geq n^{1-\varepsilon}$}
\bigr\}.
\]
By union bound and Markov's inequality, we have
%
\begin{equation}
\label{degbound} \mathbb{P} (\mathcal{D}_{n,m} )\leq \pmatrix{m
\cr 2}\mathbb{P} \bigl(d_1d_2\ge
n^{1-\varepsilon} \bigr) \le \pmatrix{m\cr 2} \frac{\mathbb{E} (d_1^{\gamma})\mathbb{E}
(d_2^{\gamma})}{n^{(1-\varepsilon)\gamma}}=o\bigl(n^{-1}
\bigr).
\end{equation}

Observe that the event $\mbox{\texttt{Solve}}_U$ implies that the
people graph induced by $U$ is connected, which in turn implies that it
contains at least $m-1$ (nonloop) edges. Partitioning on $\mathcal
{D}_{n,m}$, we have
%
\begin{eqnarray}\label{edgebound1}
&& \mathbb{P} (\mbox{\texttt{Solve}}_U )
\nonumber\\[-8pt]\\[-8pt]
&&\qquad \leq\mathbb{P} (\mathcal{D}_{n,m} ) +\mathbb{P}
\bigl(E_{\mathrm{people}}|_{U}\mbox{ has } \geq m-1 \mbox{ nonloop
edges}, \mathcal{D}_{n,m}^c \bigr).\nonumber
\end{eqnarray}

Let $\mathcal{F}_\vk:= \{d_1=k_1, \ldots, d_m=k_m\}$ be the event
that the
degrees of the vertices in $U$ are $\vk:=(k_1, \ldots, k_m)$. On the
event $\mathcal{F}_\vk$, label the half-edges at vertex $u \in U$ as $(u,1),
(u,2), \ldots, (u,k_u)$. Let
%
\begin{eqnarray}
\mathcal{E} = \mathcal{E}(\vk)&&\mbox{denote the set of all pairs of
half-edges } \bigl\{(u,\ell_u), (v, \ell_v)\bigr\}
\nonumber\\[-8pt]\\[-8pt]\nonumber
&&\mbox{such that } 1\leq u<v\leq m, 1\leq\ell_u\leq
k_u\mbox{ and }1\leq\ell_v\leq
k_v.
\end{eqnarray}
Note that $\mathcal{E}$ does not contain any pairs of half-edges that
would form a self-loop if joined.

Conditional on $\mathcal{F}_\vk$, for each $e\in\mathcal{E}$, let
$Y_{e}$ be the indicator that the half-edges in $e$ are matched in the
construction of the configuration model graph, so $E_{\mathrm
{people}}$ contains an edge between the vertices of $e$. The number of
nonloop people edges between vertices of $U$ is then $X_m = \sum_{e\in\mathcal{E}} Y_e$. By Markov's inequality, the probability of
$\{X_m\geq m-1\}$ given $\mathcal{F}_{\vk}$ is at most the expected
number of subsets of~$\mathcal{E}$ with size $m-1$ such that all
half-edge pairs in the subset get matched in the construction of the
configuration model graph. Therefore,
\[
\mathbb{P} (X_m \geq m-1 \mid\mathcal{F}_\vk
) \leq \llvert \mathcal{E}\rrvert ^{m-1} \max\mathbb{P}
(Y_{e_1}= \cdots= Y_{e_{m-1}} =1\mid\mathcal{F}_\vk
),
\]
where the maximum is taken over all subsets of size $m-1$ of
$\mathcal{E}$. If $\mathcal{F}_\vk\subseteq\mathcal{D}_{n,m}^c$,
then on the
event $\mathcal{F}_\vk$,
\[
\llvert \mathcal{E}\rrvert = \sum_{1\leq u<v\leq m}
k_uk_v \leq m^2 n^{1-\varepsilon}.
\]
For any fixed set of half-edge pairs, $e_1, \ldots, e_{m-1}\in
\mathcal{E}$, we consider the probability of matching each of these
pairs sequentially in the configuration model. Since $d_{\min}\geq3$,
each vertex outside of $U$ has at least $3$ half-edges, so each
half-edge among the first $2(m-1)$ that get matched have at least
$3(n-m) - 2(m-1) \geq n$ (for large $n$) choices for half-edges to get
matched with. Therefore,
\[
\mathbb{P} (Y_{e_1}= \cdots= Y_{e_{m-1}} =1\mid
\mathcal{F}_\vk ) \leq \biggl(\frac{1}{n} \biggr)^{m-1}.
\]
The last three displays imply that
\[
\mathbb{P} (X_m \geq m-1 \mid\mathcal{F}_\vk
) \leq m^{2m} n^{-\varepsilon(m-1)},
\]
provided $\mathcal{F}_\vk\subseteq\mathcal{D}_{n,m}^c$. Therefore,
%
\begin{eqnarray}
\label{edgebound2}
&& \mathbb{P} \bigl(E_{\mathrm{people}}|_{U}
\mbox{ has } \geq m-1\mbox{ nonloop edges, }\mathcal{D}_{n,m}^c
\bigr)
\nonumber\\[-8pt]\\[-8pt]\nonumber
&&\qquad= \sum_{\vk\dvtx  \mathcal{F}_\vk\subseteq\mathcal{D}_{n,m}^c} \mathbb{P} (
\mathcal{F}_\vk )\mathbb{P} (X_m\geq m-1\mid
\mathcal{F}_\vk ) \leq m^{2m} n^{-\varepsilon(m-1)}.
\end{eqnarray}

Choosing $m$ such that $\varepsilon(m-1) > 1$, and combining equations
(\ref{degbound}), (\ref{edgebound1}) and~(\ref{edgebound2}) show
that $\displaystyle\mathbb{P} (\mbox{\texttt{Solve}}_U ) = o(n^{-1})$.
\end{pf}

Finally we are ready to complete the proof of Proposition~\ref
{jigsawwithCM}. First, observe that the jigsaw percolation process
can be slowed down, such that at every step only a single pair of
clusters is merged. The final set of clusters after all possible merges
are made will be the same as in the original formulation, but in the
slowed down version, the size of the largest cluster can at most double
at each step. This means that for any $k \leq n/2$,
%
\begin{equation}
\label{solvebound} \mathbb{P} (\mbox{\texttt{Solve}} ) \leq\mathbb{P}
\biggl(\bigcup_{m\in[k,2k]} \bigcup
_{U\subset V, \llvert U\rrvert =m} \mbox {\texttt{Solve}}_U \biggr).
\end{equation}
Furthermore, observe that the second union on the right-hand side can
be restricted to only those subsets $U\subset V$ that are connected in
$(V,E_{\mathrm{puzzle}})$. The number of connected subsets of vertices
in $(V,E_{\mathrm{puzzle}})$ of size $m$ is crudely bounded above by
$n \cdot(m-1)! D^{m-1}$. This bound is obtained by building a
connected set $U$ of size $m$ by first choosing a starting vertex $v$,
in $n$ ways, then adding one vertex at a time to $U$ until $U$ contains
$m$ vertices. When $U$ contains $\ell$ vertices, there are at most
$\ell D$ vertices that are adjacent to a vertex in $U$ that can be
added in the next step. If we fix
$k
> 1 + \frac{2\alpha}{\alpha-2}$, then~(\ref{solvebound}) and
Lemma~\ref{solvesubsetlemma} imply that
\begin{eqnarray*}
\mathbb{P} (\mbox{\texttt{Solve}} ) & \leq&(k+1) (2k)!
D^{2k}\cdot n \cdot\max_{m\in[k,2k]} \max
_{U\subset V, \llvert U\rrvert  = m}\mathbb{P} (\mbox{\texttt{Solve}}_U
) = o(1).
\end{eqnarray*}

\section{Discussion and future directions} \label{discussion}
In our early attempts to understand jigsaw percolation on the ring
graph, we tried to use simulations to inform our conjectures about the
critical value $p_c(n)$ [Figure~\ref{figPSolven1000}(a)]. However, as
with bootstrap \mbox{percolation}~\cite{GH2008}, we expect a slow rate of
convergence to the critical value.

%
\begin{figure}[t]
\begin{tabular}{@{}c@{\hspace*{5pt}}c@{}}

\includegraphics{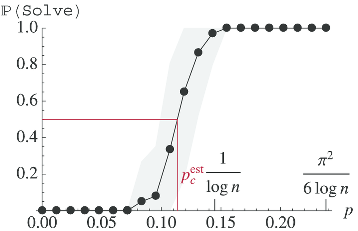}  & \includegraphics{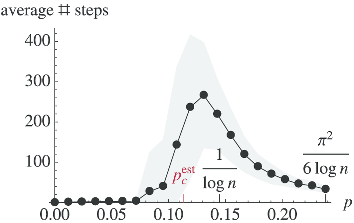}\\
\footnotesize{(a) Fraction of trials in which the people graph} & \footnotesize{(b) Average number of steps before the}\\[-1.5pt]
\footnotesize{solves the $n=1000$ ring puzzle} & \footnotesize{process stops}
\end{tabular}
\caption{Simulations of jigsaw percolation on a ring of size $n=1000$,
with 200 trials for 21 equally spaced values of $p \in[0, 1.05 \times
\pi^2/(6 \log n)]$ (which took 57 days on a department server). Dots
are averages of 200 trials, while shaded gray areas denote $\pm1$
standard deviation. The estimated critical value $p_c^{\mathrm{est}}
\approx0.11$, denoted in red, is obtained by fitting a line between
the two data points with $\mathbb{P}_{p} (\mbox{\texttt
{Solve}} )$ just below and above $1/2$. Characterizing the
average number of time steps before the process terminates
\textup{(b)} remains an open question.}\label{figPSolven1000}\label{figsteps}
\end{figure}

%
\begin{conjecture}
For jigsaw percolation on the ring puzzle graph with an Erd\H{o}s--R\'
{e}nyi people graph, there exist constants $b>0$, $c_1>0$ and $c_2$
such that
\[
p_c(n) = \frac{c_1}{\log n} + \frac{c_2}{(\log n)^{1+b}} + o \bigl((\log
n)^{-1-b} \bigr).
\]
\end{conjecture}

If true, this means that estimating $c_1$ to within $1\%$ via
simulation would require taking $n$ to be at least $\exp
[(100c_2/c_1)^{1/b}]$, which is prohibitively large if $\llvert c_2/c_1\rrvert $ is much larger than~$0.1$, and $b$ is at most $1$.
However, we expect our upper bound on $p_c(n)$ to be tight for the ring graph.

\begin{conjecture}
For jigsaw percolation on the ring puzzle graph, $c_1 = \pi^2 / 6$.
\end{conjecture}

This conjecture is based on a computation (not shown here) that implies
that a two-sided growth version of the sufficient condition used in the
proof of Proposition~\ref{ubd-ring} (i.e., the one-sided requirement
that $j$ is connected to $\{1, 2, \ldots, j-1\}$ for each $j$) yields\vspace*{1pt}
the same upper bound of $\pi^2 / (6 \log n)$ but with a correction of
order $(\log n)^{-3/2}$. Of course, even when the two-sided growth
process fails starting from every vertex, it may still be possible to
solve the puzzle by merging the clusters formed. However, if none of
these ``two-sided growth clusters'' intersect, then the puzzle is
unlikely to be solved, so we suspect that $c_1=\pi^2/6$ is the correct
lower bound.

Of particular interest for future study, the number of steps until the
process stops measures how efficiently the network solves the puzzle or
determines that it cannot be solved. We numerically simulated the
average number of steps until the process terminates for the ring
puzzle [Figure~\ref{figsteps}(b)]. As expected, the number of steps
increases around the phase transition $p_c(n)$. The process terminates
quickly when the puzzle is not solved, and the proof of
Proposition~\ref{ubd-ring} implies that the number of steps is at most
$O(\log n / p_n)$, though this is not the best bound possible. The
proof of Proposition~\ref{teoring-lower-bound} shows that for the
ring puzzle with $p_{n}\le1/(27\log n)$, the largest jigsaw cluster
(and hence number of steps) is smaller than $\log n$. As $p_n$
increases near $p_c(n)$, the puzzle may be solved, but just barely, so
the number of steps required is largest. As $p_n$ increases further,
more people-edges leads to larger clusters early in the process.
Determining the form of the function in Figure~\ref{figsteps}(b) is an
interesting open problem.

\begin{open}
\label{determinecurves}
For the ring puzzle, let $N_n$ be the smallest value of $i$ such that
$\mathcal{C}_i = \mathcal{C}_{i+1}$. Determine the asymptotic
behaviors of
\[
\mathbb{E} _{p_n} \bigl[N_n|{\mbox{\texttt{Solve}}}^c
\bigr]\quad\mbox{and}\quad E_{p_n} [N_n|{\mbox{\texttt{Solve}}}]
\]
as functions of $p_n$.
\end{open}

Finally, we suspect that the phase transition at $p_c(n)$ is sharp, in
the following sense.

\begin{conjecture}
Define $p_{\varepsilon}(n)$ as the unique $p$ for which $\mathbb
{P}_{p} (\emph{\mbox{\texttt{Solve}}} ) = \varepsilon$. Then
\[
p_{\varepsilon}(n)/p_{1-\varepsilon}(n)\to1
\]
as $n\to\infty$ for any $\varepsilon\in(0,1)$ fixed.
\end{conjecture}

Other avenues of future study include extensions and modifications of
jigsaw percolation. Different people and puzzle graphs (especially ones
with unbounded degree) are one natural direction, with mathematical and
practical interest.

\begin{open}
Consider other people and puzzle graphs, especially puzzles with
unbounded degree.
\end{open}

Another natural direction is to modify the model to make it more
realistic. For example, by analogy with the ``adjacent-edge''
modification of explosive percolation~\cite{DSouza2010}, in the
``adjacent-edge'' (AE) version of jigsaw percolation, the rule for
merging two clusters $U$ and $W$ requires that the people- and
puzzle-edges between $U$ and $W$ coincide on at least one vertex. That
is, in the AE rule, two jigsaw clusters $U$ and $W$ merge only if there
exist $u\in U$ and $w,w'\in W$ such that $(u,w)\in E_{\mathrm
{puzzle}}$ and $(u,w')\in E_{\mathrm{people}}$. In this version, a
single person must determine whether her friends' jigsaw clusters fit
with her piece of the puzzle, but she does not need to be aware of how
her entire jigsaw cluster fits with the clusters of her acquaintances.
This process is slightly more local, so we suspect that more detailed,
rigorous results are possible. Note that all of our results for jigsaw
percolation also hold for AE jigsaw percolation.

\begin{open}
Does the behavior of AE jigsaw percolation differ significantly from
that of jigsaw percolation for some class of puzzle graphs? Can more
precise statements be made about the behavior of AE jigsaw percolation
on the ring graph?
\end{open}

Another potentially interesting modification is to change the map from
people to puzzle pieces so that it is no longer bijective. This would
allow many people to have the same idea and a single person to have
multiple ideas.

%
\begin{open}
What is the effect of changing the map between people and puzzle pieces
on a network's ability to solve the puzzle?
\end{open}

In this paper, each person has one unique puzzle piece (or idea). The
critical value $p_c(n)$ marks the phase transition in the connectivity
of the Erd\H{o}s--R\'{e}nyi people graph at which it begins to solve
the puzzle with high probability. For a large class of puzzle graphs
($n$-cyles in Theorem~\ref{teo-ring}, bounded-degree puzzles in
Theorem~\ref{teo-d>1}), we show that this phase transition decreases
with $n$. However, the \emph{critical average degree}, $n p_c(n)$,
increases with the size $n$ of the social network and of the puzzle.
Thus, as social networks and the puzzles they try to solve grow
commensurately in size, people must interact with more people in order
to realize enough compatible, partial solutions. This model therefore
suggests a mechanism for the recent statistical claims that as cities
become more dense, people interact more~\cite{Schlapfer2012} and hence
innovate more~\cite{Bettencourt2007,Bettencourt2010}. Furthermore,
most social networks wish to minimize communication overhead; the
critical value $p_c(n)$ indicates the minimal communication needed to
collaboratively solve large puzzles.

Surprisingly, social networks with power-law degree distributions
lack the connectivity needed to solve bounded-degree puzzles
(Proposition~\ref{jigsawwithCM}). However, scientific
collaboration networks manage to solve puzzles despite their
heavy-tailed degree distributions~\cite
{Barabasi2002,Newman2001phys,Newman2001pnas}. This highlights the
importance of
considering more realistic assumptions in the model and of drawing
from (still nascent) studies on knowledge spaces~\cite{Chen2009}.

This work, the first step in analyzing a rich, mathematical model,
begins to suggest why certain social networks stifle creativity and why
others innovate. With a homogeneous degree distribution and
sufficiently many interactions, a social network can collectively merge
the pieces of a large puzzle---and perhaps merge the ideas that lead to
a great idea.

\section*{Acknowledgments}
We thank Rick Durrett, M.~Puck Rombach, Peter Mucha, Raissa D'Souza,
Alex Waagen, Pierre-Andr\'e No\"el and Madeleine D\"app for useful
discussions. We also thank an anonymous referee for helpful comments
that improved the presentation of the article.




%

\printaddresses
\end{document}